\documentclass[aip, amsmath,amssymb,reprint]{revtex4-1}
\usepackage{amsfonts}
\usepackage{geometry}
\usepackage{graphicx}
\usepackage{dcolumn}
\usepackage{bm}

\newtheorem{theorem}{Theorem}

\newtheorem{corollary}[theorem]{Corollary}

\newtheorem{definition}[theorem]{Definition}

\newtheorem{lemma}[theorem]{Lemma}

\newtheorem{proposition}[theorem]{Proposition}

\newenvironment{proof}[1][Proof]{\noindent\textbf{#1.} }{\ \rule{0.5em}{0.5em}}
\begin{document}

\author{Sini\v{s}a Slijep\v{c}evi\'{c}}
\affiliation{ 
Department of Mathematics, Bijeni\v{c}ka 30, Zagreb, Croatia
}
\title{Stability of synchronization in dissipatively driven
Frenkel-Kontorova models}
\date{\today}

\begin{abstract}
We rigorously show that dissipatively driven Frenkel-Kontorova models with either uniform or time-periodic driving asymptotically synchronize for a wide range of initial conditions. The main tool is a new Lyapunov function, as well as a 2D representation of the attractor. We then characterize dynamical phase transitions and outline new algorithms for determining them.
\end{abstract}

\pacs{05.45.Xt, 05.70.Fh, 63.10.+a}

\keywords{Frenkel-Kontorova model, synchronization, dynamical phase transition, attractor, space-time chaos}

\maketitle

\begin{quotation}
The model of a one-dimensional chain of particles connected by elastic springs in a spatially periodic potential, known as the Frenkel-Kontorova (FK) model, has been the paradigm tool for studying spatially modulated structures in solid state physics and 
beyond\cite{Braun:04, Floria96}. It has been numerically observed that if a FK chain is dissipatively driven with either a uniform (DC) or periodic  (AC) force, and with sufficiently strong damping, then the chain often asymptotically synchronizes. We prove that for a wide range of initial conditions and FK model parameters, by extending known techniques of monotone (or order-preserving) dynamics. Even though the model is deterministic, we study it in the statistical (or ergodic-theoretical) context, with particular focus on the notion of pinning/depinning (or dynamical Aubry) phase transitions. Our approach in particular enables applying tools of Hamiltonian dynamics to dissipative FK dynamics arbitrarily far from the equilibrium, thus extending reach of the pioneering ideas of Aubry and Mather\cite{Aubry88, Mather82}. 
\end{quotation}

\section{Introduction}

The over-damped dynamics of FK models studied here (also called {\it gradient} dynamics) has been accepted as a good approximation of physical situations with sufficiently strong damping\cite{Baesens05, Floria96}. The actual model and equations of motion are given in Section \ref{sec:one}. As reported in detail by Floria and Mazo\cite{Floria96}, the dynamical (Aubry) phase transition for DC driving is characterized by the occurrence of uniform asymptotic sliding of the chain. The situation in the AC case is more complex as further discussed in Section \ref{sec:one}, but asymptotic synchronization also often occurs. Middleton\cite{Middleton92}, Baesens and MacKay\cite{Baesens98} partially explained this as a consequence of order-preserving (or monotonicity) of the dynamics. This means that if two chain configurations $u=(u_i)_{i \in \mathbb{R}}$, $v=(v_i)_{i \in \mathbb{R}}$ are ordered, e.g. $u \leq v$ (where $\leq$ holds in each coordinate), then this ordering persists with dynamic evolution of configurations.

We show in Section \ref{sec:four} that synchronized solutions are globally attracting in the depinned phase of the dynamics and locally attracting in the pinned phase, for any initial configuration of {\it bounded width}. We thus extend already known rigorous results for spatially periodic configurations and DC setting. To do that, we propose a focus on asymptotic behavior different from the traditional. As known for example in the PDE setting\cite{Miranville08}, understanding the attractor of systems on infinite domains (that means all the asymptotics of all the initial conditions) is very difficult even for the simplest systems, as the attractor is typically infinite dimensional. Here we focus instead on asymptotics observable with non-zero (i.e. strictly positive) space-time probability, or more precisely observable for positive density of spatial translates and time evolutions. We call the set of such configurations the {\it space-time attractor}, and define it in Section \ref{sec:two}.

We then observe in Section \ref{sec:three} that the space-time attractor of our model is two-dimensional. This enables us to describe it in some detail. For example, we show that in the depinned phase, the attractor consists entirely of synchronized solutions.

Our analysis naturally leads to the ergodic-theoretical setting, and to the study of invariant probability measures (invariant with respect to both the time evolution and spatial translations). This results with characterizations of dynamical Aubry phase transitions. Physically, the pinned phase has been understood as the phase where parts of the physical space are asymptotically "off-bounds", while in the depinned phase the chain can slide over the entire space. This is related to analyticity/non-analyticity of modulation functions and various other model features. We give a precise definition of this understanding, and show that this is equivalent to statistical definitions of dynamical phases, related to uniqueness/non-uniqueness of space-time invariant measures.

Finally, we discuss applications of tools from Hamiltonian dynamics to dissipative FK dynamics. Aubry and Mather\cite{Aubry88,Mather82} successfully applied these ideas to the description of equilibria of the FK model (i.e. without driving), which can be characterized as orbits of a symplectic map (an area-preserving twist diffeomorphism\cite{Katok95}). We show that the space-time attractor arbitrarily far from equilibrium can be characterized in a similar way. As an example of an application of this, we then outline how the Converse KAM theory can be used to determine dynamical phase transitions.

We give rigorous mathematical proofs to all the statements in the paper. For easier reading, most of the proofs have been moved to the Appendix at the end of the paper. A detailed (and quite technical) proof of the main tool, the Theorem \ref{t:project}, has already been reported in Ref. \onlinecite{Slijepcevic14} (we outline the core of the argument here). The results on stability of synchronization and characterization of phase transitions, as well as applications, are new.

\section{Setting and numerical background} \label{sec:one}

\subsection{The model}

Consider a set of particles in one dimension, denote position of each by 
a real number $u_{j}$ and the configuration of the entire chain as $u=(u_{j})_{j\in \mathbb{%
Z}}$. The energy of the generalized FK model can be formally defined as
\begin{equation}
H=\sum_{j\in \mathbb{Z}}\left( W(u_{j+1}-u_{j})-V(u_{j})\right).
\label{r:fk}
\end{equation}%
Here $V(u)$ is a periodic on-site potential (i.e. $V(u+1)=V(u)$), and $W(p)$ is
a generalized elastic coupling, by which we mean a strictly convex function (i.e. such that $W^{\prime \prime
}\geq \delta >0$ for some $\delta >0$). The standard FK model
is defined by particular functions 
\begin{eqnarray*}
V(u) &=& -k\cos (2\pi u)/(2\pi )^{2}, \\
W(p) &=& (p-\mu )^{2}/2
\end{eqnarray*}
We focus here on the dissipative, overdamped (also called gradient) dynamics, given by the equations%
\begin{eqnarray}
\frac{d}{dt}u_{j}(t) &=&-\frac{\partial }{\partial u_{j}}H(u)+f(t), 
\nonumber \\
\frac{d}{dt}u_{j}(t) &=&W^{\prime }(u_{j+1}-u_{j})-W^{\prime }(u_{j}-u_{j-1})
\nonumber \\
&&+V^{\prime }(u_{j})+f(t).  \label{r:eq}
\end{eqnarray}

The driving force $f(t)$ can be constant, in which case we consider \textit{%
DC} dynamics of the FK model. Alternatively, $f$ can be time-periodic (\textit{%
AC} dynamics). As we can reparametrize the time, we can in the AC
case assume that $f(t+1)=f(t)$.

We summarize the standing assumptions on the model (\ref{r:eq}):

\begin{description}
\item[(A)] $W$ is $C^{2}$, strictly convex, such that $W^{\prime \prime
}\geq \delta >0;$ $V$ is $1$-periodic;\ in the AC\ case \thinspace $W,V,f\,\ 
$are real analytic.
\end{description}

\subsection{Ground states and synchronization}

We first briefly recall the structure of the ground states of the chain (\ref%
{r:fk}), independently described by Aubry and Mather\cite{Aubry88,Mather82}. 
First note that all the equilibria of (\ref{r:fk}),
that is the configurations $u=(u_{j})_{j\in \mathbb{Z}}$ which solve $(\ref{r:eq})$
with $f(t)=\partial u_{j}(t)/\partial t=0$, can be interpreted as orbits of a
2-dimensional map. The Aubry-Mather theory
focuses on ground states (as a subset of equilibria) defined as follows. 
As the total energy $H$ of the infinite chain is
typically infinite, the ground states are defined as configurations for
which the energy of any finite subsegment of the chain $%
(u_{m},u_{m+1},...,u_{n})$ is minimal if we fix positions of end particles $%
u_{m},u_{n}$ and allow all others to arbitrarily vary. Importantly, each
ground state has a well defined \textit{mean spacing} 
\[
\rho (u)=\lim_{n-m\rightarrow \infty }(u_{n}-u_{m})/(n-m).
\] 
Furthermore, one can find
a ground state for any rational (\textit{commensurate configurations}) or
irrational (\textit{incommensurate configurations}) $\rho (u)$ (Ref. 
\onlinecite{Bangert88}, Theorems 3.16 and 3.17), thus the structure of ground states is
quite rich.

Important tools when studying ground states, as well as driven dynamics, are based
on considering ordering and intersection of configurations. 
We first recall the definition of spatial translations $S_{m,n}$ of configurations (defined for any integers $m,n$):
\[
S_{m,n}u_{j}=u_{j-m}+n.
\]
If $T_{t}$, $t \geq 0$ is the time evolution of (\ref{r:eq}), that means
$T_{t}u(s)=u(t+s)$, then by definition $S_{m,n}$ and $T_{t}$ commute.
We say that two configurations $u,v$ intersect if their graphs (as
functions $j\mapsto u_{j}$) intersect; more precisely if for some $j$, 
$(v_{j+1}-u_{j+1})(v_j-u_{j})\leq 0$ (but $u$, $v$ not equal). 

The operators $T$, $S$ enable us to precisely define synchronized solutions of 
(\ref{r:eq}). We consider a solution $u(t)$ synchronized if the trajectory
of each particle is time-periodic (where we identify $u$ and $u+n$ for interger $n$), and the trajectory
of each particle coincides (up to a shift in phase). We introduce an equivalent definition of a synchronized 
solution in terms of intersection of configurations, which will be very useful in the following.

\begin{definition} \label{d:one}
We say that a solution $u(t)$ of (\ref{r:eq})\ is synchronized, if 
for any integers $m,n,s$ and any $t \in \mathbb{R}$, $u(t)$ 
and $S_{m,n}u(t+s)$ do not intersect.
\end{definition}

\noindent (In the definition we implicitly assume that $u(t)$ exists for all times $t \in \mathbb{R}$.)
An immediate consequence is that all the spatial and temporal translates of a
synchronized solution can be represented as a one parameter family of
configurations. If we identify $u$ and $u+n$ for all
integers $n$ (as we will often do in the following), $S,T$ translates of a synchronized
solution can be parametrized by a subset of a circle. Elementary results of the theory of
one-dimensional dynamical systems (the Denjoy theory, Ref. \onlinecite{Katok95}, Section 12) then
imply that they typically (i.e. for irrational $\rho$) either cover an entire
circle, or its Cantor subset.

Ground states are important examples of synchronized solutions, as shown by the Aubry-Mather theory (Ref. \onlinecite{Bangert88}, Theorem 3.13). Note that for ground states, $T_t$ is constant, thus synchronization is equivalent to non-intersection of spatial translates.

In general, it is rigorously known that synchronized solutions exist. This has been (partially very recently)
proved by Baesens, MacKay and Qin in the DC case\cite{Baesens98,Qin10,Qin11},
and by Qin in the AC case\cite{Qin13}:

\begin{theorem}
\label{t:synchronized} Assuming (A), there exists a synchronized
solution $u(t)$ of (\ref{r:eq}) for any (AC or DC) forcing $f(t)$ and any (rational or
irrational) mean spacing $\rho \in \mathbb{R}$.
\end{theorem}

\subsection{Numerical observations} \label{sub:numerical}

\begin{figure}
\includegraphics[width=\linewidth]{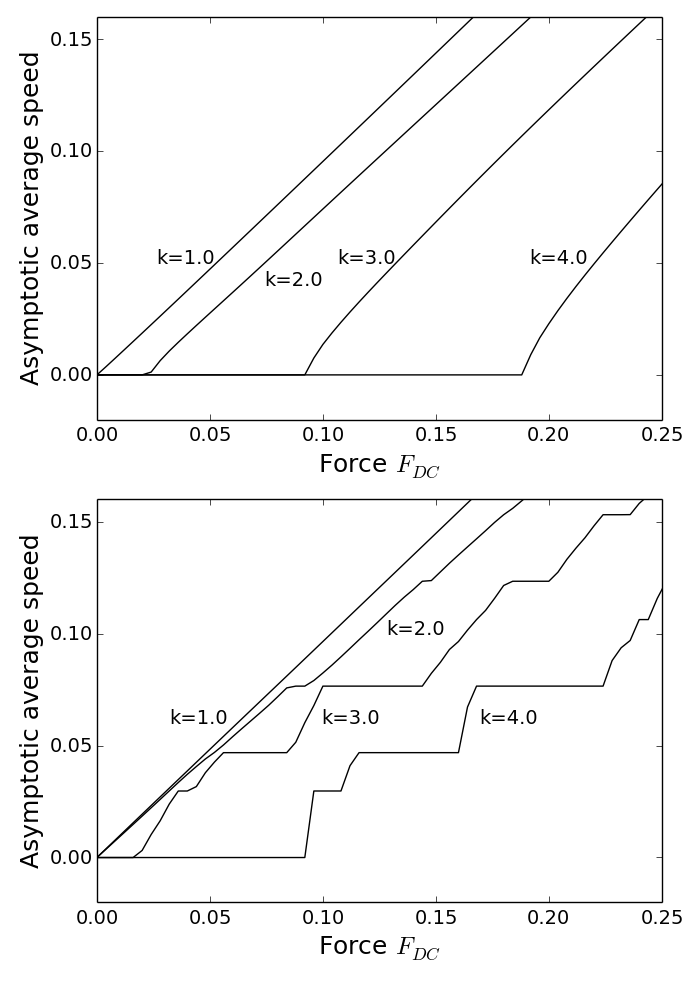}
\caption{\label{fig:dc1}The $v/F_{DC}$ dependency for a DC- (above) and AC- (below) driven standard FK chain,
for mean spacing $\rho=245/397$ and different values of $F_{DC}, k$. In the AC case, $F_{AC}=0.2$, $\nu_0=0.2$}
\end{figure}

Numerical simulations\cite{Floria96,Mazo95} showed 
that synchronization often asymptotically appears
when forcing $f(t)$ is switched on (numerically an infinite chain is
approximated with a long finite one with periodic boundary conditions). In
the DC case, it was seen that for any initial condition the dynamics is
attracted to a synchronized solution, called \textit{uniformly sliding
solution,} as long as the asymptotic average sliding speed is not
zero. The asymptotic average sliding speed $v(\rho)$ depends (continuously) only on
the forcing $f$ and the mean spacing $\rho$. Numerics further showed that
the critical \textit{depinning force} is not zero if and only if the set of
ground states with the mean spacing $\rho $ does not project in any particle
coordinate to the entire real line. 

As an example, we consider
the driven standard FK model given by the equations
\begin{eqnarray}
\frac{du_j(t)}{dt} &=& u_{j+1}-2u_j+u_{j-1}+\frac{k}{2\pi }\sin (2 \pi u_j)+f(t)  
\nonumber \\
f(t)&=& F_{DC}+F_{AC}\sin (2 \pi \nu_0 t), \nonumber
\end{eqnarray}
\noindent with parameters $k, F_{DC}, F_{AC}, \nu_0$.

A typical dependence of $v(\rho)$ on $F_{DC}$ in the DC case ($F_{AC}=0$), 
and in the AC case for a fixed $F_{AC}, \nu_0$ is in
Figure \ref{fig:dc1}.

In the AC case, $v(\rho)$ also depends continuously on $\lambda $ and $\rho $,
but is not smooth up to a certain critical value of $F_{DC}$ - the 
dynamical Aubry transition. Also it was seen 
that as long as the asymptotic speed is not zero, 
the solution typically asymptotically synchronizes. 

\section{Mathematical background} \label{sec:two}

\subsection{Asymptotics of a solution}

In this paper we consider configurations of \textit{bounded width}, by which we mean
configurations $u$ for which there exists mean spacing $\rho (u)$ and a real
number such that for some constant $K>0$ and all integers $m,n$,%
\begin{equation}
|u_{m}-u_{n}-(m-n)\rho (u)|\leq K.  \label{r:bWidth}
\end{equation}%
(With some technical care as done in Ref. \onlinecite{Slijepcevic14} but beyond the scope of this paper, all the results
also hold on more general space of configurations of bounded spacing, that
is satisfying $\sup_{m\in \mathbb{Z}}|u_{m+1}-u_{m}|<\infty $). Denote by $%
\widetilde{\mathcal{X}}$ the space of all the configurations of bounded width, and with\ $%
\widetilde{\mathcal{X}}_{\rho}$ its subspace of configurations 
with the mean spacing $\rho $. For example, as by the standard result of Aubry-Mather theory all the ground states
satisfy (\ref{r:bWidth}) with $K=1$ (Ref \onlinecite{Bangert88}, Corollary 3.16), they are in $\widetilde{\mathcal{X}}$. By
the same argument, the synchronized configurations also satisfy (\ref{r:bWidth}) with $K=1$ and are in $\widetilde{\mathcal{X}}$.

Standard results of existence of ODE on Banach spaces imply that (\ref{r:fk}) 
generates a smooth semiflow on $\widetilde{\mathcal{X}}$, with $\widetilde{\mathcal{X}}_{\rho}$
being invariant sets\cite{Baesens98,Baesens05,Slijepcevic14}. 
When considering asymptotics, we consider pointwise convergence of
configurations (i.e. the product topology on $\widetilde{\mathcal{X}}$), rather than
uniform convergence.

The usual notions of dynamical systems theory make sense only for relatively
compact trajectories. Fortunately this holds for initial conditions in 
$\widetilde{\mathcal{X}}_{\rho}$ if we identify $u$ and $u+n$ for all integer $n$. We denote $\mathcal{X}$, $\mathcal{X}_{\rho}$ to be the quotient sets with respect to that relation, and omit the subscript $n$ in $S_{m,n}$.
Now each $\mathcal{X}_{\rho}$ is compact and invariant for both spatial
translations $S_m$ and the time evolution $T_t$, $t\geq 0$ (invariance
follows from the order preserving property, Ref. \onlinecite{Slijepcevic14}, Section 4).

The usual notion of $\omega $-limit set considers all the limit points of a
trajectory as time goes to infinity. Here we consider a typically smaller set of 
physically the most relevant asymptotic trajectories, which are asymptotically 
observed with non-zero probability with respect to time and spatial
windows. A precise definition follows:

\begin{definition}
The weak $\omega $-limit set of $u$, denoted by $\widetilde{\omega}(u)$, is the
smallest closed set such that for any (arbitrarily small)\ open neighborhood 
$U$ of $\widetilde{\omega}(u)$, the ratio of $m,n$, $0\leq m\leq N$, $%
-N\leq n\leq N$ for which $T_{m}S_{n}u\in U$, converges to $1$ as $N \rightarrow \infty $.
\end{definition}

An equivalent definition of $\widetilde{\omega}(u)$ is given in Lemma \ref{l:charomega} in the Appendix.
It is easy to show that $\widetilde{\omega}(u)$ is a well-defined, closed set
and a subset of $\omega (u)$\cite{Slijepcevic14}. Unlikely the $\omega$-limit set, 
$\widetilde{\omega}$ is not necessarily connected. An interpretation of $\widetilde{\omega}$
is that, if we choose randomly a (sufficiently large) time, and a random,
arbitrarily large, spatial window of an infinite chain, we will with
asymptotic probability $1$ observe a configuration in $\widetilde{\omega}(u)$.
We propose the definition of a (space-time) attractor for a spatially extended (i.e. on
an infinite domain) system like (\ref{r:eq}) to be the set $\mathcal{A}$
which is the closure of the union of all weak $\omega $-limit sets, 
$$\mathcal{A}=\text{Cl} \left( \bigcup _{u\in \mathcal{X}}\widetilde{\omega}(u) \right).$$

Our definition of $\mathcal{A}$ results with a typically smaller set than attractor as standardly defined\cite{Miranville08}. We think our definition is physically relevant. 
One can say that the standard attractor incorporates all the asymptotics of the
dynamics, while our set $\mathcal{A}$ captures asymptotics with non-zero space-time positive probability. Furthermore, if for
example a system like (\ref{r:eq}) is space-time chaotic in the sense of
Bunimovich and Sinai\cite{Bunimovich88}, one can show that the space-time
chaos is then contained in $\mathcal{A}$\cite{Slijepcevic14a}.

\subsection{Invariant measures and dynamical phase transitions}

In addition to considering evolution of individual configurations with
respect to (\ref{r:fk}), we find useful considering simultaneous evolution
of a family of initial configurations and finding average
properties of this evolution. We make it precise by considering evolution of
probability measures, and more specifically probability
measures invariant for the spatial shift. 

We can write $S=S_{1}$, $T=T_{1}$ (the time-one map for (\ref{r:eq}%
)). Note that there exists a
huge number of $S$-invariant probability measures on the state space $%
\mathcal{X}$. For example, we can embed in $\mathcal{X}$ in many ways the
standard Bernoulli probability measure on the space of bi-infinite sequences
of $\{0, 1\}$. More generally, a $S$-invariant
probability measure is any (shift-invariant)\ random process which
constructs a configuration in $\mathcal{X}$.

Given any $S$-invariant measure $\mu$, we can consider its evolution $\mu
(t)$ with respect to (\ref{r:fk})\ by considering evolution of each
configuration (mathematically, $\mu (t)$
is the pulled measure $\mu (t)=T_{t}^{\ast }\mu (0)$). A $T$-invariant
measure is a measure which is also invariant with respect to the evolution (%
\ref{r:eq}). The importance of $S,T$-invariant measures on $\mathcal{X}$ is
in the following fact:

\begin{proposition}
\label{p:attractor}The attractor $\mathcal{A}$ coincides with the union of
supports of all $S,T$-invariant measures on $\mathcal{X}$. Furthermore, for
any $\rho \in \mathbb{R}$, $\mathcal{A}_{\rho }=\mathcal{A\cap X}_{\rho }$
is not empty.
\end{proposition}

\noindent (We postpone the proof of this as most the other other claims to the Appendix.)
Thus the study of the attractor $\mathcal{A}$ is equivalent to understanding
the structure of $S,T$-invariant measures. 

We can now extend the Definition \ref{d:one} of synchronized configurations and solutions
to measures. We say that a {\it synchronized measure} is a $S,T$-invariant probability measure such that
no two configurations in its support intersect. Clearly by definition, a synchronized 
measure is supported on synchronized trajectories. Denote by $\mathcal{S}_{\rho }\subset \mathcal{A}_{\rho }$ the union of
supports of all the synchronized measures on $\mathcal{X}_{\rho}$. We will see
in the next section that $\mathcal{S}_{\rho }$ is not empty. 

One typically distinguishes two dynamical phases of (\ref{r:eq}), depending on whether
transport is possible. In the depinned phase, asymptotically each
particle can slide over the entire $\mathbb{R}$, while in the pinned phase
some of the regions are off-limit as the force is too weak as compared to
the potential $V$ and the related Peierls-Nabarro barrier\cite{Floria96}. 
We propose a rigorous way to define pinned vs. depinned phase, which works both
in the DC and AC case, by using the language of
measures, in the spirit of Mather\cite{Mather89}. 

Let $\pi_0 : {\cal X} \rightarrow \mathbb{S}^1$ be the projection of a configuration $u$
to $u_0$ (it projects onto the circle $\mathbb{S}^1=\mathbb{R}/\mathbb{Z}$, as we identify
$u$ and $u+n$ for integer $n$).

\begin{definition} \label{d:two}
We say that (\ref{r:eq}) is in the depinned phase for a given mean spacing $\rho \in \mathbb{R}$, 
if there exists a synchronized measure $\mu$ with the mean spacing $\rho$ such that $\pi_0(\textnormal{supp}({\mu }))$ 
is onto (i.e. the entire $\mathbb{S}^1$); 
otherwise it is in the pinned phase.
\end{definition}

\noindent Here $\textnormal{supp}(\mu)$ denotes the support of $\mu$. 
Note that, as we consider $S$-invariant measures, this definition
is independent of the projected coordinate. For a given
one-parameter family of FK-chains or forces $f(t)$, the dynamical Aubry
transition for a given $\rho \in \mathbb{R}$ is the value of the parameter
in which the pinned/depinned phase changes.

Equivalence of the definition of depinned/pinned phases as above to analyticity (resp. non-analyticity) of modulation functions, as well as to the dependence of the average speed on average driving force as described in subsection \ref{sub:numerical},
was essentially shown by Qin\cite{Qin11,Qin13}.

Analogously to thermodynamics, one can expect that the difference in phases of (\ref{r:eq}) is
whether there is a unique (ergodic) $S,T$-invariant measure with a chosen mean spacing or not. We will see later that this is
indeed a characterization if $\rho$ is irrational (also for rational $\rho$ with additional technical restrictions): 
the $S,T$-invariant measure is unique in and strictly in the 
depinned phase.

\subsection{Dynamics of intersections of solutions}

The key feature of the dynamics (\ref{r:eq}), that it is order-preserving,
was used by Middleton, Baesens and MacKay\cite{Baesens98,Baesens05,Middleton92},
in proving results on asymptotics of (mostly) finite
chains with DC dynamics. We will here use a generalization of that
related to counting intersections of two solutions. We already introduced the notion of intersection of two
configurations. It is important to distinguish transversal and
non-transversal intersections, as in Figure 2 (see Ref. \onlinecite{Slijepcevic14} for
a precise classification).
 
\begin{figure}
\includegraphics[width=\linewidth]{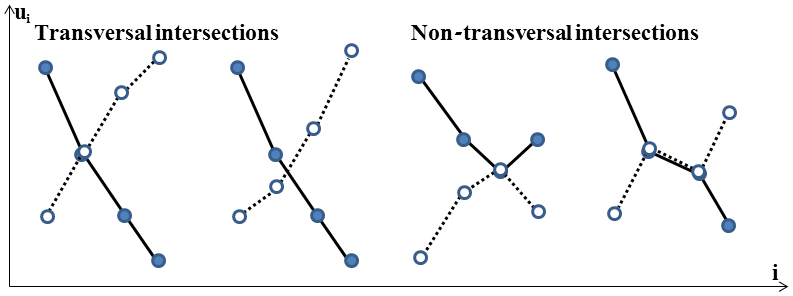}
\caption{\label{fig:intersections}Examples of transversal and non-transversal intersections of configurations.}
\end{figure}

A stronger version of the
order-preserving rule is that, if $u(0)$
and $v(0)$ are two chain configurations with at most finite number of
intersections, then the following is known and can be
proved by considering linearization of (\ref{r:eq}) (Ref. 
\onlinecite{Slijepcevic14}, Sections 3 and 4):

\begin{description}
\item[(I1)] The number of intersections of $u(t)$ and $v(t)$ is a
non-increasing function of $t$,

\item[(I2)] If $u(t)$ and $v(t)$ intersect non-transversally at $t_{0}$,
then the number of intersections strictly drops at $t=t_{0}$.
\end{description}

These ideas originate from 1D parabolic PDE's, where they have been
extensively used to describe their asymptotics\cite{Angenent88a,Fiedler89,Joly10}.
They are, however, not
directly applicable to the dynamics (\ref{r:eq}) (or also to PDE's on
unbounded domains), as two arbitrary configurations $u,v$ (with the same
mean spacing)\ typically intersect infinitely many times. We resolve this by
considering \textit{an average number of intersections }with respect to a
probability measure on the state space $\mathcal{X}$. Precisely, let $%
\mathcal{I}(u,v)\in \{0,1\}$ be the function which assigns $1$ if $u,v$
intersect in the interval $[0,1)$, otherwise it is $0$. If $\mu \,\ $is any $%
S$-invariant probability measure on $\mathcal{X}$, we define the average
number of self-intersections of $\mu $ as
\begin{equation}
\mathcal{I}(\mu )=\sum_{n\in \mathbb{Z}}\int \mathcal{I}(u+n,v)d\mu (u)d\mu (v)\text{.}
\label{r:selfintersect}
\end{equation}

The meaning of (\ref{r:selfintersect}) is the following: the expression $%
\int I(u,v)d\mu (u)d\mu (v)$ is the probability that two randomly chosen
configurations $u,v$ intersect in the interval $[0,1)$. As $\mathcal{\mu }$
is invariant for the spatial shift $S$, this is also the probability of
finding an intersection in any interval $[m,m+1)$. We have a sum over $n$ in
the definition of $\mathcal{I(\mu)}$ to make sure it is well defined on the
quotient space $\mathcal{X}$, as we
identify configurations $u$ and $u+n$. It is easy to
check $\mathcal{I(\mu )}$ is always finite\cite{Slijepcevic14}. One can
now show that $\mathcal{I(\mu )}$ has properties mimicking (I1), (I2),
without any restrictions to chosen configurations and measures. A rigorous
proof of the following is in Ref. \onlinecite{Slijepcevic14}; we sketch the argument in the Appendix.

\begin{theorem}
\label{t:lyapunov}If $\mu $ is an $S$-invariant probability measure with
evolution $\mu (t)$ with respect to (\ref{r:eq}), then:

\begin{description}
\item[(M1)] The function $t\mapsto \mathcal{I(}\mu (t)\mathcal{)}$ is
non-increasing;

\item[(M2)] If for some $t_{0}$, there are $u$, $v$ in the support of $\mu
(t_{0})$ with a non-transversal intersection, then $t\mapsto \mathcal{I(}\mu
(t)\mathcal{)}$ is strictly decreasing at $t_{0}$ (i.e. $\mathcal{I(}\mu
(t+\varepsilon )\mathcal{)<I(}\mu (t-\varepsilon )\mathcal{)}$ for any $%
\varepsilon >0$).
\end{description}
\end{theorem}

The property (M1) means that $\mathcal{I}$ is a Lyapunov function on the
space of $S$-invariant measures. As for synchronized measures, $\mathcal{I}$
reaches its minimum zero (see the comment after Theorem \ref{t:project} below), 
these measures are expected to be Lyapunov stable.
We deduce implications of this to asymptotics of individual trajectories in
Section \ref{sec:four}.

\section{2D representation of the attractor} \label{sec:three}

\begin{figure*}
\includegraphics[width=\linewidth]{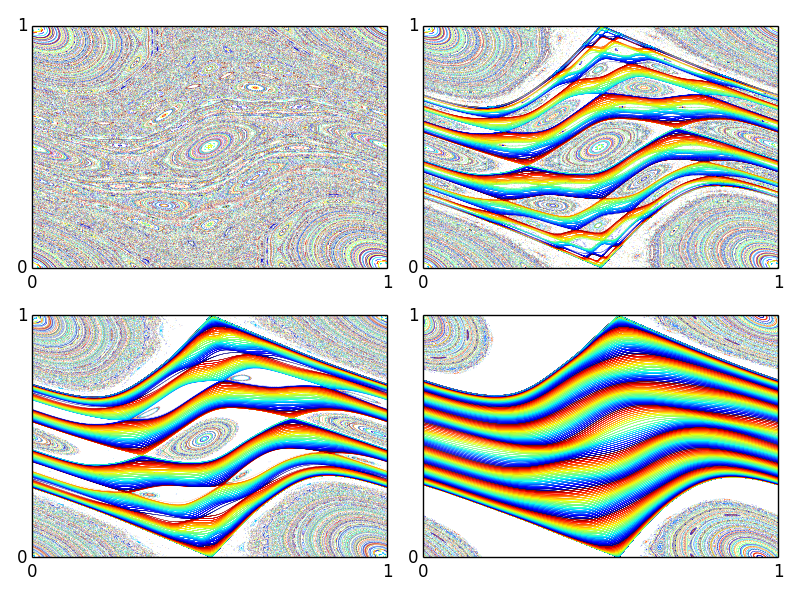}
\caption{\label{fig:representation}2D representations of the attractor of a DC-driven standard FK chain with $k=1.0$.
The DC force (left to right): $F=0$, $0.001$, $0.005$, $0.05$. The same color corresponds to the same configuration (i.e. the same orbit with respect to $h$ acting on $\mathbb{S}_{1}\times \mathbb{R}$) and its time evolution.}
\end{figure*}

We now show that the attractor $\mathcal{A}$ of the dynamics (\ref{r:eq})
can be represented as a $2$-dimensional map, which will be important for
both qualitative and quantitative description of the dynamics. This fact
has been extensively studies in the case $f=0$, as noted in Section \ref{sec:one}. It is
somewhat unexpected that this principle extends for arbitrary DC and AC
forcing. We define the projection $\pi :\mathcal{A}\rightarrow \mathbb{S}^1 \times \mathbb{R}$ with
\[
\pi (u)=(u_{0},u_{1}-u_{0})\text{.} 
\]
\begin{theorem}
\label{t:project}Any two configurations $u,v\in \mathcal{A}$ can not
intersect non-transversally. Furthermore, $\pi $ is injective.
\end{theorem}

This follows directly from Theorem \ref{t:lyapunov} (details in the Appendix).
A direct consequence is that synchronized measures are characterized as $S,T$-invariant
measures $\mu$ for which $\mathcal{I}(\mu)=0$. Here is why: if $\mathcal{I}(\mu)=0$, then
by continuity the only possible intersections in the support of $\mu$ are non-transversal, which is
impossible.

Projection $\pi$ enables us now to visualize and analyze $\mathcal{A}$ in
2D, as done in Figure \ref{fig:representation}, plotting images of the projected spatial shift map $%
h=\pi \circ S\circ \pi ^{-1}$ in the same color. As we identified orbits $u$
and $u+n$ for integer $n$, the map $h$ is well defined on the cylinder $%
\mathbb{S}_{1}\times \mathbb{R}$, and can be understood as a dynamical
analogue of an area-preserving twist map on the cylinder which describes
the attractor in the case with no force.

We will see that the approximately level circles ("rainbows" in Figure \ref{fig:representation})
correspond to depinned
synchronized trajectories, with the average coordinate $p$ corresponding to
the mean spacing $\rho$. We call them {\it KAM-circles} (borrowing terminology from the Hamiltonian dynamics), and define them as homotopically non-trivial
(i.e. not compressible to a fixed point) invariant circles of the function $%
h $ on $\mathbb{S}_{1}\times \mathbb{R}$. Analysis of the pinned/depinned
phase transition will rely on the 2D representation as in Figure  \ref{fig:representation}, and will
use the following:

\begin{theorem}
\label{c:ground} Given any mean spacing $\rho \in \mathbb{R}$, the set
$\mathcal{S}_{\rho }$ is not empty. In the depinned phase, $\mathcal{A}_{\rho }=\mathcal{S}_{\rho }$
and it projects to a KAM-circle. \end{theorem}

The converse of the last statement in Theorem \ref{c:ground} also holds. Given a map $h$ and a point
$x=\pi(u) \in \mathbb{S}_{1}\times \mathbb{R}$, one can recover the mean spacing 
of $u$ by calculating the {\it rotation number} of $x$ with respect to $h$, defined
as the average of the first coordinate of $h$-iterates of $x$. By adapting the proof of
Theorem \ref{c:ground} given in the Appendix, one can also show that
if (\ref{r:eq}) is the pinned phase for a given $\rho$, then $h$ does not have a
KAM-circle with that rotation number. We omit the proof.

An implication is that the structure of synchronized orbits, when projected
to the cylinder, is analogous to the structure of ground states when
projected to the cylinder. Thus we can use tools from Hamiltonian dynamics
as outlined in the next section.

We now come back to the definition of the pinned vs. depinned phase and its
characterization.

\begin{corollary} \label{c:unique}
If $\rho$ is irrational and (\ref{r:eq}) in the depinned phase, then 
there is a unique ergodic $S,T$-invariant measure with the mean spacing $\rho$.
\end{corollary}

The converse of Corollary \ref{c:unique} also holds, and can be shown by 
adapting the construction of Mather's connecting orbits of area-preserving
twist maps\cite{Mather93,Slijepcevic99} and related invariant measures. This would result with a rich family
of ergodic $S,T$-invariant measures with the same (pinned) rotation number;
we omit details of the construction.

Finally, the following fact
regarding intersection of depinned synchronized configurations will be the key in applications.

\begin{corollary}
\label{l:intersect}If (\ref{r:eq})\ is in the depinned phase for $\rho \in 
\mathbb{R}$, given any $u\in \mathcal{A}_{\rho }$, no configuration in $%
\mathcal{A}$ can intersect $u$ more than once.
\end{corollary}

\section{Synchronized orbits are attracting} \label{sec:four}

Attractiveness and global stability (in the sense of definitions in Section
\ref{sec:two}, i.e. ignoring probability/density 0 times and space windows) of
synchronized orbits in the depinned phase is now straightforward.

\begin{corollary}
\label{t:depinned}Assume (\ref{r:eq}) is in the depinned phase for a given $%
\rho \in \mathbb{R}$. Then for any $u\in \mathcal{X}_{\rho }$, $\widetilde{\omega%
}(u)$ consists of (depinned)\ synchronized orbits, i.e. $\widetilde{\omega}%
(u)\subset \mathcal{S}_{\rho }$.
\end{corollary}

\begin{proof}
By definition, $\widetilde{\omega}(u)\subset \mathcal{A}%
_{\rho }$, and by Theorem \ref{c:ground}, $\mathcal{A}_{\rho }=\mathcal{S}%
_{\rho }$.
\end{proof}

\vspace{1ex}

In the pinned phase, one can not expect such general results. For example,
even in the stationary case $f(t)=0$, the structure of the pinned part of $%
\mathcal{A}$ is quite complex (it is analogous to understanding Birkhoff
regions of instability of area-preserving twist maps, whose complexity is
still not fully understood). We can, however, describe a relatively rich
family of configurations in the basin of attraction of $\mathcal{S}_{\rho }$: ground states of FK model
with defects ({\it discommensurations}), that is with missing or squeezed in extra particles within the
ground state structure\cite{Braun:04,Floria96}.
The following abstract definition generalizes this notion: we say that a
configuration $u$ has $k$ defects, if $k$ is the maximal number of
intersections of $u$ and $S_{m,n}u$ over all integers $m,n$.

\begin{theorem}
\label{t:density} In the DC case, if $u\in \mathcal{X}_{\rho }$ has finitely
many defects, then $\widetilde{\omega}(u)\subset \mathcal{S}_{\rho }$.
\end{theorem}

The proof of that in the Appendix could be also extended to hold in the AC
case and for configurations with zero defect density. We intend to provide
details on that separately.

\section{Conclusion} \label{sec:five}

\subsection{Numerical determination of dynamical phase transitions} 

As an example of an application of tools from Hamiltonian dynamics enabled by our
2d representation of the attractor, we consider tools known as {\it Converse KAM}.
These algorithms focus on break-up of KAM-circles and more generally KAM tori. 
They have beem developed in the context of area-preserving twist maps:
by MacKay and Stark\cite{MacKay85} (based on
earlier works of Mather, Herman and others);\ Boyland and Hall\cite{Boyland87}, 
and Greene\cite{Greene79}. The first two are based on a simple
characterization: KAM-circles are "barriers for transport", which,
expressed in terms of intersections of configurations, states that no
configuration (associated to an orbit)\ can intersect a KAM-configuration
more than once. As this is by Corollary \ref{l:intersect} indeed a property of
depinned synchronized configurations, the Converse KAM algorithms should be
applicable.

For example, the Boyland-Hall algorithm can be applied to determining the dynamical phase
transition in the DC phase as follows. Assume $u$ is a stationary periodic configuration of (\ref{r:fk})
of the type $(p,q)$; that means $T_t(u)=u$ for all $t$, and $u_{k+q}=u_{k}+p$ for all $k$. We define its rotation
band $r(u)=(r^{-}(u),r^{+}(u))$, where%
\begin{eqnarray*}
r^{-}(u) &=&\min_{j,k=1,...,q}\frac{\left\lceil u_{j+k}-u_{j}\right\rceil }{k%
}, \\
r^{+}(u) &=&\max_{j,k=1,...,q}\frac{\left\lfloor u_{j+k}-u_{j}\right\rfloor 
}{k}.
\end{eqnarray*}

\noindent Here $\left\lceil x\right\rceil $ is the smallest integer larger than $x$,
and $\left\lfloor x\right\rfloor $ the largest integer smaller than $x$.
Let ${\cal R}$ be the union of all the rotation bands for all the stationary periodic configurations.

The algorithm is based on the following Theorem

\begin{theorem}
The system (\ref{r:eq}) with DC driving is in the pinned phase for a given
mean spacing $\rho \in \mathbb{R}$, if and only if $\rho \in {\cal R}$.
\end{theorem}

The proof is an application of Corollary \ref{l:intersect} and techniques from 
Ref. \onlinecite{Boyland87} and will be reported in Ref. \onlinecite{Nincevic14}.
As we can numerically find stationary periodic configurations 
(by finding extremal points of a "tilted" energy functional $H$
restricted to the periodic configurations) and calculate
their rotation bands, the phase diagram can be calculated with arbitrary precision.

\subsection{Perspectives}

One can question whether the results on infinite chains are only of
theoretical interest, as all systems in nature are finite. We argue that
analyzing such \textit{extended} dynamical systems is the right tool to
obtain bounds (for example on relaxation times) independent of the system
size. For example in Ref. \onlinecite{Gallay14} we used that approach to find bounds on
relaxation times of the viscous fluid turbulence independent of the
reservoir size, and in Refs. \onlinecite{Gallay:01}, \onlinecite{Gallay:12}, to many other
systems. Our results imply that for FK\ models, there exist bounds on
convergence times to synchronized solutions independent of the chain size.
It is an important next step to find them.

The description of the dynamics of driven FK\ models here is not complete,
as we described only the "non-zero space-time probability" asymptotics. A
complete asymptotics description is most likely a combination of the 2D,
typically synchronized, dynamics we described, and coarsening as described
by Eckmann and Rougemont\cite{Eckmann98} for a similar system. That means
that different parts of the chain converge to different dynamical ground
states, and then they either diffusively or in a sequence of discrete
coarsening "jumps" coallesce on larger and larger space and time scales. A
more precise description of this dynamics would be nice.

The results here apply also to more general 1D chains whose energy is given
by a function $V(u_{j},u_{j+1})$ satisfying the twist condition $\partial
_{xy}V(x,y)\leq -\delta <0$, as long as the interaction is only between the
nearest neighbours. For longer range interactions, even if the dynamics is
cooperative (i.e. the off-diagonal elements of the linearized equation are
positive), the intersection-counting tools do not apply\cite{Slijepcevic14}. 
The approach is, however, applicable to the
ratchet dynamics of FK chains (no driving; but either the site potential
or the interaction potential change periodically in time), where exact
results are scarce\cite{Floria05}. Furthermore, it applies also to the second order dynamics as
long as damping is strong enough, and to analogous continuous space
systems\cite{Slijepcevic14a}.

Finally, we propose focusing on the attractor as defined here, or
equivalently on description of the space-time invariant measures, when
studying any extended dynamical system (by that we mean lattice systems of
infinitely many ODE\cite{Slijepcevic13}; or PDE on unbounded domains\cite{Gallay:01,Gallay:12}). This should result
with, for example, better understanding of existence and frequency of
occurrence of the space-time chaos, as our understanding of that phenomenon is quite limited\cite{Bunimovich88,Mielke09}.

\appendix*

\section{Proofs of Theorems}

In all the proofs we use the fact that $\mathcal{X}_{\rho }=\widetilde{\mathcal{X}}_{\rho }/\sim $
is compact in the (always implicitly assumed) induced product topology, which follows from
the Tychonoff theorem. Here $\widetilde{\mathcal{X}}$, $\widetilde{\mathcal{X}}_{\rho }$ are the sets of 
configurations of bounded width as defined in Section III.A, $\sim $ is the relation of equivalence, $u\sim
u+n$ for any integer $n$. As in most of the paper, we abuse the notation and denote
by $u$ both configurations in $\mathcal{X}$ and their representatives in $\widetilde{\mathcal{X}}$.
All the measures are Borel probability measures on 
$\mathcal{X}_{\rho }$ or $\mathcal{X}$. By the Alaoglu theorem, the set of
probability measures on $\mathcal{X}_{\rho }$ is compact in the weak* topology.
We always assume without loss of generality that an ($T$- or $S$-invariant) 
measure on $\mathcal{X}$ is actually supported on $\mathcal{X}_{\rho }$
for some $\rho \in \mathbb{R}$. We can do that by the Ergodic Decomposition
Theorem\cite{Katok95}, as $\mathcal{X}_{\rho }$ is $S$- and $T$-invariant. All the
proofs are done by considering the invariance
with respect to the time-one map $T=T_{1}$. They could be
easily adjusted to the DC case with invariance being considered with
respect to the semiflow $T_{t}$, $t\geq 0$.

\vspace{1ex}
Prior to the proof of Proposition \ref{p:attractor}, we introduce characterization 
of $\widetilde{\omega}$-limit sets. Denote by $\mathbf{1}_{\mathcal{V}}$ 
the characteristic function of a set $\mathcal{V}$.

\vspace{1ex}

\begin{lemma} \label{l:charomega} A configuration $v \in \widetilde{\omega}(u)$ if
only if for each open neighbourhood $\mathcal{V}$ of $v$, there exists $\delta >0$ and a sequence 
$N_k \rightarrow \infty$ such that 
\[
\frac{1}{(N_k+1)(2N_{k}+1)}\sum_{n=0}^{N_{k}}%
\sum_{m=-N_{k}}^{N_{k}} \mathbf{1}_{\mathcal{V}}({T_{n}S_{m}v}) \geq \delta.
\]
\end{lemma}
\begin{proof} Straightforward, by contradiction. \end{proof}

\vspace{1ex}

\begin{proof}[Proof of Proposition \protect\ref{p:attractor}]
We first show that each $v\in \widetilde{\omega}(u)$ is in the support of some $%
S,T$-invariant measure. Choose an open neighborhood $\mathcal{V}$ of $v\in \widetilde{\omega}(u)$, 
and find $\delta >0$ and a sequence 
$N_k \rightarrow \infty$ as in Lemma \ref{l:charomega}. 
Let $\mu_k$ be a sequence of measures on $\mathcal{X}_{\rho }$
defined as
\[
\mu_k=\frac{1}{(N_k+1)(2N_{k}+1)}\sum_{n=0}^{N_{k}}%
\sum_{m=-N_{k}}^{N_{k}}\delta _{T_{n}S_{m}u},
\]
where $\delta _{u}$ is the Dirac measure supported on $u$. It is easy to
check that the limit $\nu$ of each convergent subsequence of $\mu_k$ (which exists due to
compactness) is a $S,T$-invariant measure on $\mathcal{X}_{\rho }$.

Now choose a sequence of decreasing open neighborhoods $\mathcal{V}_n$ of a fixed
$v\in \widetilde{\omega}(u)$ such that $\cap_{n \in \mathbb{N}} \mathcal{V}_n = \{ v \}$
and construct a sequence of associated $S,T$-invariant measures $\nu_n$ as above.
Then $v$ is in the support of the $S,T$-invariant measure 
$\sum_{n \in \mathbb{N}} (1/2^n) \nu_n$.

The converse follows from the fact that the union of supports of all the 
$S,T$-invariant
measures is closed (Ref. \onlinecite{Slijepcevic14a}, Lemma 2.1) and a standard ergodic-theoretical
argument based on the Birkhoff Ergodic Theorem and the Ergodic Decomposition Theorem.
As logically not required for the results that follow, we omit the details.

The set 
$\mathcal{A}_{\rho }$ is not empty, as $\mathcal{X}_{\rho}$ is compact, $S,T$%
-invariant, thus it supports a $S,T$-invariant measure (Ref. \onlinecite{Slijepcevic14},
Lemma 7.2).
\end{proof}

\vspace{1ex}

\begin{proof}[Outline of proof of Theorem \protect\ref{t:lyapunov}]
A detailed proof is given in Ref. \onlinecite{Slijepcevic14}, Propositions 5 and 6; we
outline the key argument. Consider the function
\[
\mathcal{I}(\mu(t))=\sum_{n \in \mathbb{Z}} \int \mathcal{I}(u(t),v(t)+n)d\mu (u(0))d\mu (v(0)),
\]

\noindent counting intersections of $u,v$ with respect to the time evolution of a
measure $\mu$. An intersection of $u(t)$, $v(t)$ can be represented by a continuous curve $\gamma (t)$,
which is defined as the point where the graph of $i \mapsto u_i(t)-v_i(t)$ crosses the $x$-axis.
As by property (I1), the function $\mathcal{I}(u(t),v(t)+n)$ is 
non-increasing (except in the cases when $\gamma (t)$ crosses the boundaries of $[0,1]$ which can be 
shown to by $S$-invariance of $\mu$ cancel out), $\mathcal{I}(\mu(t))$ is non-increasing.
Similarly, a continuity argument and (I2) imply (M2).
\end{proof}

\vspace{1ex}
\begin{proof}[Proof of Theorem \protect\ref{t:project}]
By Proposition \ref{p:attractor}, we can assume $u$, $v$ are in the support
of some $S,T$-invariant measure $\mu $ (if not the same measure, we take their convex
combination $\mu =\mu _{1}/2+\mu _{2}/2$ and obtain again a $S,T$-invariant
measure). If $u$, $v$ intersect non-transversally, by Theorem \ref{t:lyapunov}%
, $\mathcal{I}(\mu (t))$ is strictly decreasing at $t=0$. But $\mu$ 
is $T$-invariant, thus $\mathcal{I}(\mu (t))$ must be constant,
which is a contradiction. Now, if $\pi (u)=\pi (v)$ for some $u,v\in 
\mathcal{A}$, by definition they intersect non-transversaly, which is
impossible.
\end{proof}

\vspace{1ex}
Prior to the proof of Theorem \ref{c:ground}, we need an intermediate result.
Assume (\ref{r:eq}) is in the depinned phase for some $\rho \in \mathbb{R}$. 
Denote by $\mathcal{S}^*_{\rho}$ the support of the synchronized measure $\mu$
from Definition \ref{d:two}.

\begin{lemma} \label{l:insert}
If $u \in \mathcal{S}^*_{\rho}$, depinned, then
no configuration in $\mathcal{A}$ can intersect $u$ more than once.
\end{lemma}

\begin{proof}
By definition, $\pi_0 : \mathcal{S}^*_{\rho} \rightarrow \mathbb{S}^1$
is bijective (injective due to synchronization; onto as in the depinned phase).
Thus its lift $\widetilde{\mathcal{S}}^*_{\rho}$ (i.e. not considered as a quotient space) 
is an image of a continuous curve $\gamma :\mathbb{R}\rightarrow \widetilde{\mathcal{A}}$, $\gamma (0)=u$, increasing
in each coordinate. Assume some $v\in \widetilde{\mathcal{A}}$ intersects $u$ twice, say
between $j,j+1$ and $k,k+1$. Without loss of generality let $v_j \leq u_j$, $%
v_{j+1}>u_{j+1},...,v_{k-1}>u_{k-1},v_{k}\geq u_{k}$, $v_{k+1}<u_{k+1}$. 
Let $s>0$ be the largest $s$ such that $\gamma (s)$ and $v$ intersect between $j,...,k+1$. By
continuity of $\gamma $, $\gamma (s)$ and $v$ intersect non-transversally,
which is in contradiction to Theorem \ref{t:project}.
\end{proof}

\vspace{1ex}

\begin{proof}[Proof of Theorem \protect\ref{c:ground}]
Let $\rho \in \mathbb{R}$. We first show that $\mathcal{S}_{\rho }$ is not empty.
Choose a synchronized orbit $u$ in $\mathcal{X}_{\rho }$ (which exists by
Theorem \ref{t:synchronized}). Let $\mathcal{B}$ be the smallest, closed, $%
S,T$-invariant set containing $u$. As (\ref{r:bWidth}) holds with $K=1$, we see that
$\mathcal{B}\subset \mathcal{X}_{\rho }$. Thus $\mathcal{B}$ is compact. 
By definition and continuity, if two configurations
in $\mathcal{B}$ intersect, they must intersect non-transversally. 
Compactness of $\mathcal{B}$ implies that there exists a $S,T$-invariant measure $\mu $
supported on $\mathcal{B}$. As $\mathcal{B}$ contains no transversal
intersections, $\mu $ by Theorem \ref{t:project} contains no intersections at all.
It is then by definition synchronized, thus $\mathcal{S}_{\rho }$ is not empty.

To show that $\mathcal{A}_{\rho }=\mathcal{S}_{\rho }$, it suffices to show that
$\mathcal{A}_{\rho }=\mathcal{S}^*_{\rho}$. Here $\mathcal{S}^*_{\rho}=\textnormal{supp}(\mu)$, $\mu$ are
as defined prior to Lemma \ref{l:insert}. Choose any $v\in \mathcal{A}_{\rho }$. By Proposition \ref{p:attractor}, $v$ is in the
support of a $S,T$-invariant measure $\zeta $. By definition of $\mu$, we can find $u$
in the support of $\mu $ such that
$u_{0}=v_{0}$. If $u\not=v$, by Lemma \ref{l:intersect} $u,v$ intersect
exactly once. Now $\mu \times \zeta $ is a probability
measure on $\mathcal{A}_{\rho } \times \mathcal{A}$,
invariant for $S\times S$. If $U,V$ are small enough neighborhoods of $u,v$ in 
$\mathcal{A}_{\rho }$ such that any two configurations in $U,V$ also
intersect at site $0$, by Poincar\'{e} recurrence applied to $S\times S$
on $U\times V$, we can find $u^{\prime }\in U$, $v^{\prime }\in V$ such that 
$u^{\prime },v^{\prime }$ intersect infinitely many times. As by (\ref%
{r:bWidth}), there exists $K$ such that for all $j\in \mathbb{Z}$, $%
|u_{j}-v_{j}|\leq K$ (as they have the same mean spacing), we can find their
representatives in $\widetilde{\mathcal{A}}_{\rho }$ which intersect infinitely many
times, which is in contradiction to Lemma \ref{l:insert}.

Now as $\mathcal{A}_{\rho }=\mathcal{S}_{\rho}=\mathcal{S}^*_{\rho}$, by definition of the depinned phase and $\mathcal{S}^*_{\rho}$
it must project to a KAM-circle.
\end{proof}

\vspace{1ex}

\begin{proof}[Proof of Corollary \protect\ref{c:unique}]
Uniqueness of the $S,T$-invariant measure for irrational $\rho$
follows from the Denjoy theory (in that case there is an unique $S$-invariant measure, 
Ref. \onlinecite{Katok95}, Section 12.7).
\end{proof}

\vspace{1ex}

\begin{proof}[Proof of Corollary \protect\ref{l:intersect}]
In the proof of Theorem \ref{c:ground} above we showed that
in the depinned phase, $\mathcal{A}_{\rho }=\mathcal{S}^*_{\rho}$.
Thus the claim follows from Lemma \ref{l:insert}.
\end{proof}

\vspace{1ex}

In the proof of Theorem \ref{t:density} we will require the following,
which was proved in Ref. \onlinecite{Slijepcevic14}, Theorem 1.4.

\begin{theorem}
\label{t:description}In the DC case, the attractor $\mathcal{A}$ consists of
equilibria and depinned synchronized trajectories.
\end{theorem}

Here equilibria are configurations $u$ for which $\partial
u/\partial t=0$. 

\vspace{1ex}

\begin{proof}[Proof of Theorem \protect\ref{t:density}]
If $u$ has at most $k$ defects, $u$ and any translate $S_{m,n}u$
intersect at most $k$ times. By continuity, as $\widetilde{\omega}(u)\subset
\omega (u)$ and as trajectories in $\widetilde{\omega}(u)$ by Theorem 
\ref{t:project} can not intersect non-transversally, for any $v\in \widetilde{\omega}%
(u)$, $v$ and $S_{n,m}v$ intersect at most $k$ times. By Theorem \ref%
{t:description}, $v$ is either depinned synchronized (so the proof is done),
or an equilibrium. If $v$ is an equilibrium, it is in the support of a 
$S$-invariant measure $\mu $ supported on the closure (in $\mathcal{X}_{\rho}$) of $S_{n}v$, $n\in \mathbb{Z}$, thus supported on equilibria. As the number of
self-intersections of two configurations in the support is finite and bounded by $k$, one can
(e.g. by the Birkhoff Ergodic Theorem) easily show that $\mathcal{I(\mu )=}0$. 
Thus all $v$ in the support of $\mathcal{\mu}$ must be synchronized.
\end{proof}


\begin{thebibliography}{99}

\bibitem{Angenent88a} S. Angenent and B. Fiedler, The dynamics of rotating
waves in scalar reaction diffusion equations, Trans. Am. Math. Soc. 307
(1988), 545-568.

\bibitem{Aubry88} S. Aubry and P. Y. Le Daeron, The discrete
Frenkel-Kontovora model and its extensions, Physica D \textbf{8} (1983),
381.422.

\bibitem{Baesens98} C.\ Baesens, R.\ S. Mackay, Gradient dynamics of tilted
Frenkel-Kontorova models, Nolinearity \textbf{11} (1998), 949-964.

\bibitem{Baesens05} C. Baesens, Spatially extended systems with monotone
dynamics (continuous time). \textit{Dynamics of coupled map lattices and of
related spatially extended systems}, 241-263, Lecture Notes in Phys. 671,
Springer (2005).

\bibitem{Bangert88} V. Bangert, Mather sets for twist geodesics on tori,
Dynamics Reported \textbf{1} (1988), 1-55.

\bibitem{Boyland87} P. Boyland and G. Hall, Invariant circles and the order
structure of periodic orbits in monotone twist maps, Topology \textbf{113}
(1987), 21-35.

\bibitem{Braun:04} O. M. Braun and Y. S. Kivshar\textit{, The
Frenkel-Kontorova Model}. Concepts, methods, applications. Springer-Verlag,
Berlin, 2004.

\bibitem{Bunimovich88} L. Bunimovich and Y. Sinai, Spacetime chaos in
coupled map lattices, Nonlinearity 1 (1988), 491-516.

\bibitem{Eckmann98} J.-P. Eckmann and J. Rougemont, Coarsening by
Ginzburg-Landau dynamics, Comm. Math.\ Phys. \textbf{199} (1998), 441-470.

\bibitem{Fiedler89} B. Fiedler, J. Mallet-Paret, A Poincar\'{e}-Bendixson
theorem for scalar reaction diffusion equations, Arch. Rational Mech.\ Anal. 
\textbf{107}\ (1989), 325-345.

\bibitem{Floria96} L. M. Floria and J. J. Mazo, Dissipative dynamics of the
Frenkel-Kontorova model, Advances in Physics, \textbf{45 }(1996), 505-598.

\bibitem{Floria05} L. M. Floria, C. Baesens and J.\ G\'{o}mez-Gardenez, The
Frenkel-Kontorova model, in \textit{Dynamics of Coupled Map Lattices and of
Related Spatially Extended Systems, }Lecture Notes in\ Phys. 671, Springer,
Berlin, 2005, 209-240.

\bibitem{Gallay:01} Th.\ Gallay and S. Slijep\v{c}evi\'{c}, Energy flow in
formally gradient partial differential equations on unbounded domains. J.
Dynam. Differential Equations \textbf{13} (2001), 757-789.

\bibitem{Gallay:12} Th.\ Gallay and S. Slijep\v{c}evi\'{c}, Distribution of
Energy and Convergence to Equilibria in Extended Dissipative Systems, to
appear in J. Dynam.\ Differential Equations.

\bibitem{Gallay14} Th.\ Gallay and S. Slijep\v{c}evi\'{c}, Uniform
boundedness and long-time asymptotics for the two-dimensional Navier-Stokes
equations in an infinite cylinder, to appear in J. Math. Fluid Mech.

\bibitem{Greene79} J. M. Greene, A method for determining a stochastic
transition, J. Math. Phys. \textbf{20} (1979), 1183.

\bibitem{Hu05} B. Hu, W. Qin and Z. Zheng, Rotation number of the overdamped
Frenkel-Kontorova model with ac-driving, Physica D, 208\ (2005), 172-190.

\bibitem{Joly10} R.\ Joly, G. Raugel, Generic Morse-Smale property for the
parabolic equation on the circle, Ann. Inst. H. Poincar\'{e}\ \textbf{27 }%
(2010)\textbf{, }1397-1440\textbf{.}

\bibitem{Katok95} A. Katok and B. Hasselblatt, \textit{Introduction to the
Modern Theory of Dynamical Systems}.\ Cambridge University Press, 1995.

\bibitem{MacKay85} R. S.\ MacKay and I. Percival, Converse KAM:\ theory and
practice, Commun. Math. Phys. 98 (1985), 469-512.

\bibitem{Mather82} J. Mather, Existence of quasi-periodic orbits for twist
homeomorphisms of the annulus, Topology \textbf{21} (1982), 457-467.

\bibitem{Mather89} J. Mather, Minimal measures, Comm. Math.\ Helvetici 
\textbf{64} (1989), 375-394.

\bibitem{Mather93} J. Mather, Variational construction of connecting orbits,
Ann. Inst. Fourier, Grenoble \textbf{43} (1993), 1349-1386.

\bibitem{Mazo95} J. J. Mazo, F. Falo and L. M. Floria, Stability of
metastable structures in dissipative ac dynamics of Frenkel-Kontorova
models, Phys. Rev. B, \textbf{52} (1995), 6451-6457.

\bibitem{Middleton92} A. A. Middleton, Asymptotic uniqueness of the sliding
state for charge density waves, Phys. Rev. Lett. \textbf{68} (1992), 670-673.

\bibitem{Mielke09} A. Mielke and S. Zelik, Multi-pulse evolution and space-time chaos in dissipative systems, Mem. Amer. Math. Soc. \textbf{198} (2009), 97pp.

\bibitem{Miranville08} A. Miranville, S. Zelik, Attractors for dissipative
partial differential equations in bounded and unbounded domains, Handbook of
differential equations:\ evolutionary equations.\ Vol. IV, 103-200,
Elsevier/North-Holland, Amsterdam, 2008.

\bibitem{Nincevic14} M. Nin\v{c}evi\'{c}, B. Rabar and S. Slijep\v{c}evi\'{c},
Converse KAM theory and dynamical phase transitions of Frenkel-Kontorova models,
in preparation.

\bibitem{Qin10} Wen-Xin Qin, Dynamics of the Frenkel-Kontorova model with
irrational rotation number, Nolinearity \textbf{23} (2010), 1873-1886.

\bibitem{Qin11} Wen-Xin Qin, Existence and modulation of uniform sliding
states in driven and overdamped particle chains, Comm. Math. Physics \textbf{311} (2011), 513-538.

\bibitem{Qin13} Wen-Xin\ Qin,\ Existence of dynamical hull functions with
two variables for the ac-driven Frenkel-Kontorova model, J. Diff. Equations 
\textbf{255} (2013), 3472-3490.

\bibitem{Slijepcevic99} S. Slijep\v{c}evi\'{c}, Monotone gradient dynamics
and Mather's shadowing, Nonlinearity \textbf{12} (1999), 969-986.

\bibitem{Slijepcevic13} S. Slijep\v{c}evi\'{c}, The energy flow of discrete extended gradient systems, Nonlinearity \textbf{26} (2013), 2051-2079.

\bibitem{Slijepcevic14} S. Slijep\v{c}evi\'{c}, The\ Aubry-Mather theorem
for driven generalized elastic chains, Disc. Cont. Dyn. Systems A \textbf{34}
(2014), 2983-3011.

\bibitem{Slijepcevic14a} S. Slijep\v{c}evi\'{c}, Entropy of scalar
reaction-diffusion equations, to appear in Math. Bohemica.

\end{thebibliography}
\end{document}